\documentclass[a4paper,11pt]{amsart}
\usepackage{amssymb, amsmath}
\usepackage{amscd}
\numberwithin{equation}{section}
\usepackage{epsfig}
\usepackage{amsmath}
\usepackage{amsfonts,amssymb,amsopn}
\usepackage[all]{xy}

\newcommand{\cc}{\mathbb{C}}

\newcommand{\F}{\mathbb{F}}

\newtheorem{theorem}{{\textbf Theorem}}
\newtheorem{proposition}[theorem]{{\textbf Proposition}}

\newtheorem{lemma}[theorem]{{\textbf Lemma}}

\newtheorem{defn}[theorem]{{\textbf Definition}}

\newtheorem{remit}[theorem]{{\textbf Remark}}

\title[Constructions by ruler and compass, together with a fixed conic]{Constructions by ruler and compass, \\together with a fixed conic}

\author{Seungjin Baek, Insong Choe, Yoonho Jung, Dongwook Lee and Junggyo Seo}

\begin{document}

\begin{abstract}
It is well-known to be impossible to trisect an arbitrary angle and duplicate an arbitrary cube by a ruler and a compass. On the other hand, it is known from the ancient times that these constructions can be performed when it is allowed to use several conic curves.
In this paper, we prove that any point constructible from conics can be constructed using a
 ruler and a compass, together with a single fixed non-degenerate conic  different from a circle.

\end{abstract}

\keywords{conic-constructible points, ruler and compass, conics   }

\subjclass[2010]{51M15(primary), 12F05(secondary)}

\maketitle

\section{introduction}

Trisecting an arbitrary angle and doubling  the cube by  ruler and compass are two of  the famous problems posed by Greeks which are known to be unsolvable. The first proof of the impossibility for the geometric constructions is attributed to Wantzel (1837). For  history on this subject, we refer the reader  to \cite[pp. 25-26]{S}. 

It is also well known that the trisection and duplication are possible if it is allowed to use one or more conic sections in addition to ruler and compass(\cite{V}). The old constructions used hyperbolas and parabolas as conic sections. Recently, Hummel \cite{H} and Gibbins--Smolinsky \cite{GS} independently found the constructions using ellipses.

Gibbins and Smolinsky, at the end of their paper,  asked if one can reduce the number of types of ellipses  for the involved constructions. More specifically Hummel, in the final remark of his paper,  asked if all points constructible from ellipses are  constructible using the same ellipse in addition to a ruler and a compass. The goal of this paper is to show that the answer is Yes.  More generally, we prove that every conic-constructible point can be obtained using a ruler and a compass, together with a single fixed non-degenerate conic  different from a circle (Theorem \ref{main}). It can be  pointed out that our result resembles the Poncelet-Steiner's theorem which states that any construction with a ruler and a compass can be accomplished by a ruler together with a fixed circle and its center(see \cite[\S3.6]{E}).

The precise meaning of the conic-constructibility will be reviewed in the next section. In \S3, we will state the main result and give a proof.

\section{Conic-constructible points}
The point $(x,y)$ in the plane $\mathbb{R}^2$ is identified with a complex number $x + i y \in \mathbb{C}$. 
Starting from the initial set $P=P_0$ of points,  perform the drawings:

(i) Given two points in $P$, draw a line through the two points using a ruler.

(ii) Given two points $z_1$ and $z_2$ in $P$, draw a circle centered at $z_1$ with radius $|z_1-z_2|$ using a compass.
  
Let $P_1$ be the set of points by adjoining all the intersections of  lines and circles from the above drawings.  Now, replacing $P_0$ by $P_1$ and running the above process, we get the set $P_2$. Repeating this  inductively, we get a sequence of sets
\[
P= P_0 \subset P_1 \subset P_2 \subset \cdots.
\]
A point is called  \textit{constructible from $P$} if it is inside $P_\infty := \bigcup_{i=0}^\infty P_i$. A point is called \textit{constructible} if it is constructible from the initial set $\{ 0,1 \}$ in $\cc$. It is well known that a point is constructible if and only if it lies in a subfield $F$ of $\cc$  which has  a finite sequence of subfields starting from rational numbers:
\[
\mathbb{Q} = F_0 \subset F_1 \subset F_2 \subset \ldots \subset F_n = F
\]
such that $[F_{i+1}: F_i] = 2$ for each $i$. 

Now in addition to the above drawings (i) and (ii), consider another one:

(iii) Draw all parabolas, ellipses, and hyperbolas having foci in $P$, directrix lines which are obtained from (i), and eccentricities equal to the length of a segment $\overline{z_1 z_2}$ for some $ z_1, z_2 \in P$. 

Let $P=Q_0$ and let $Q_1$ be the set of points by adjoining all the intersections of  lines and circles and conics from the  drawings (i), (ii), and (iii). Replacing $Q_0$ by $Q_1$ and running the same kind of process, we get $Q_2$.  Again by induction, we get a sequence of sets
\[
P= Q_0 \subset Q_1 \subset Q_2 \subset \cdots.
\]
A point is called  \textit{conic-constructible from $P$} if it is inside $Q_\infty := \bigcup_{i=0}^\infty Q_i$. A point is called \textit{conic-constructible} if it is conic constructible from the initial set $\{ 0,1 \}$ in $\cc$.

Note that for any $P$, both the  set of  constructible points derived from $P$ and the set of  conic-constructible points derived from $P$ are subfields of $\cc$, because the process (i) and (ii) are already enough to produce the complex numbers given by the operations $+, -, \times, \div$. The $x$-coordinates of the constructible (or conic-constructible) points derived from $P$ again form a subfield of $\mathbb{R}$, which are called \textit{constructible (or conic-constructible) numbers derived from $P$}.
Videla found a useful criterion:
\begin{proposition} \label{videla} {\rm (\cite[Theorem 1, 2]{V})} \\ 
(1) The set of conic-constructible points forms the smallest subfield of $\cc$ containing 0, 1, and $i$ which is closed under conjugation, square roots and cube roots.\\
(2) A point in $ \cc$ is conic-constructible if and only if it is contained in a subfield of $\cc$ of the form $\mathbb{Q}(\alpha_1, \alpha_2, \ldots, \alpha_n)$, where $\alpha_1^{k_1} \in \mathbb{Q}$ and $\alpha_i^{k_i} \in \mathbb{Q}(\alpha_1, \ldots, \alpha_{i-1})$ for  $2 \le i \le n$,  where  $k_i  \in \{2,3\}$ for each $i = 1, 2, \ldots, n$. \qed
\end{proposition}

The point of Proposition \ref{videla} (2)  is that taking square and cube roots of a given complex number are precisely the constructions required to obtain all the conic-constructible points.  In other words, providing the constructions for real cube roots and trisections of angles are sufficient to get all conic-constructible points.

Let $K$ be a subfield of $\mathbb{R}$ such that every positive number $x \in K$ has a square root in $K$.  A line passing through two points of $K^2$ is called a \textit{line in $K$}. A circle is called a \textit{circle in $K$} if its center is in $K^2$ and it passes through a point in $K^2$. Similarly, a conic (parabola, ellipse, or hyperbola)  is called a \textit{conic in $K$} if its foci are in $K^2$, its directrix line is in $K$, and its eccentricity is in $K$.
Recall the following useful fact.
\begin{proposition}  \label{useful}  \  Let $E$ be a non-degenerate conic (different from a circle), defined by the equation $ax^2 + bxy +cy^2 +dx+ ey+f = 0$. Then $E$ is  in $K$ if and only if the coefficients $a,b, c, d, e, f$ are in $K$ (after suitable rescaling).  
\label{fact} \end{proposition}
\begin{proof}
The proof for the case of ellipse is given in \cite[Proposition 3]{GS} and the   argument can  be modified to work also for parabolas and hyperbolas. 
\end{proof}

\section{result}

We introduce two more  notions: $C$-constructible points and $e$-constructible points.   Let $C$ be a non-degenerate conic  in the field  of constructible numbers.  For the initial set $P \subset \cc$, a point is \textit{$C$-constructible from $P$} if it is  obtained by the same process as conic-constructible points except that the conics in  step (iii) are confined to the fixed one: $C$.

 We call an ellipse or a hyperbola of eccentricity $e>0$ to be ``regular'' if  it is given by the equation
\[
(1-e^2) (x-a)^2 + (y-b)^2 = \lambda^2
\]
for some $a, b, \lambda \in \mathbb{R}$. Also, we call a  parabola to be regular if it is of the form
\[
x = \lambda(y-a)^2 +b
\]
for some $a, b, \lambda \in \mathbb{R}$.

Let $e>0$ be a constructible number. A point is \textit{$e$-constructible from $P$} if it is obtained by the same process as conic-constructible points except that 

(1) the conics in  step (iii) are confined to the regular ones of eccentricity $e$, and

(2) the intersections of two conics, neither of which is a circle, are not adjoined. \\
In other words, (1) a conic is drawn only when it is regular with eccentricity $e$, and (2) for any drawn conic different from a circle, only the intersections with lines and circles are counted for the next stage.
\\
As before when $P = \{0,1 \}$, we say $C$-constructible and $e$-constructible respectively, for short. By definition, if $C$ is a regular conic with eccentricity $e$, then
\begin{center}
$C$-constructible \ $\Rightarrow$ \ $e$-constructible \ $\Rightarrow$ \ 
conic-constructible.
\end{center}

\begin{lemma} \label{e}
Let $e>0$ be any constructible  number.
Every conic-constructible point is $e$-constructible.
\end{lemma}
\begin{proof}
 Let $z \in \cc$ be a conic-constructible point. By Proposition \ref{videla}, $z \in  \mathbb{Q} (\alpha_1, \alpha_2, \ldots, \alpha_n)$, where either $\alpha_i^2$ or $\alpha_i^3$ is contained in $ \mathbb{Q} (\alpha_1, \ldots, \alpha_{i-1})$ for each $i$.
We will show that $z$ is $e$-constructible by induction on $n$.  Suppose that all the points in $H:= \mathbb{Q} (\alpha_1, \alpha_2, \ldots, \alpha_{n-1})$ are $e$-constructible.
 If $\alpha_n^2 \in H$, then $\alpha_n$ is constructible from $H$. 
 
 Now suppose $\alpha_n^3 =  re^{i \theta} \in H$.  Let $q = \cos \theta$ and let $K$ be the field of constructible numbers derived from $0,1$ and $r, q$.

First consider the intersection of the following circle and conic:
\[
x^2 + y^2 - rx - y = 0 \ \ \text{and} \ \ (1-e^2)x^2 + y^2  - rx - (1-e^2)y  = 0  .
\]
By Proposition \ref{fact}, these are a circle and a conic in $K$.
From these two equations, we get $x^4 - rx = 0$. Hence the $x$-coordinate of the intersection points other than the origin corresponds to the cube root of $r$. 

Next for $q = \cos \theta$, by triple-angle formula, $\cos \left(  {\theta}/{3} \right)$ is a real solution  to the equation 
\[
4x^3 - 3x - q = 0.
\]
To get this equation, we intersect the following circle and conic:
\[
x^2 + y^2 - \frac{q}{4}x - \frac{7}{4}y = 0 \ \ \text{and} \ \ (1-e^2)x^2 + y^2- \frac{q}{4}x - \left(\frac{7}{4} - e^2 \right)y = 0  .
\]
Again, these are a circle and a conic in $K$. 
Note that the conics in the above are regular of eccentricity $e$. 
From these intersections, we constructed the numbers $\sqrt[3]{r}$ and $ \cos (\theta/3)$, and eventually the point $\alpha_n$. Thus so far we have shown that $\alpha_n$ is an $e$-constructible point derived from 0, 1, and $r, q$. Since we assumed that all the points in $H$ are $e$-constructible and $\alpha_n^3 = re^{i \theta} \in H$, $r$ and $q= \cos \theta$ are also $e$-constructible. 
Therefore, $\alpha_n$ is $e$-constructible and thus all the points in $\mathbb{Q}(\alpha_1, \alpha_2, \ldots, \alpha_{n})$ are $e$-constructible.
\end{proof}

Now comes our main result: Every conic-constructible point can be constructed by using a single fixed  conic in addition to a ruler and a compass.

\begin{theorem} \ Let $C$ be any non-degenerate conic, different from a circle, in the field of constructible numbers.
Then every  conic-constructible point is $C$-constructible.
\label{main} \end{theorem}

\begin{proof} Let $e>0$ be any constructible number.
By Lemma \ref{e}, it suffices to prove that every $e$-constructible point is $C$-constructible.
We first introduce a hierarchy on the $e$-constructible points starting from the field of constructible points $\F_0$:
\[
\F_0 \subset \F_1 \subset \F_2 \subset \cdots .
\]
Draw all the regular conics in $\F_0$ of eccentricity $e$. Let  $Q_{1}^e$ be the set of points by adjoining to $\F_0$ all the intersections of lines and circles with any of the  drawn regular conics. Let $\F_1$ be the field of constructible points derived from $Q_1^e$.  

Inductively, for each $k \ge 0$, 
let $Q_{k+1}^e$ be the set of points by adjoining to $\F_k$ all the intersections of lines and circles with any of the drawn regular conics in $\F_k$ of eccentricity $e$. And define $\F_{k+1}$ as the field of constructible points derived from $Q_{k+1}^e$. Then the set of $e$-constructible points coincides with $\bigcup_{k=0}^\infty \F_k$.

Now we prove that every $e$-constructible point is  $C$-constructible by induction on $k$. Certainly the constructible points $\F_0$ are $C$-constructible. Now assume that all the points in $\F_k$ are $C$-constructible. Take any point $z \in \F_{k+1} \setminus \F_k$ which is obtained as an intersection of a circle $R_k$ in $\F_k$ and a regular conic $C_k$ in $\F_k$ of eccentricity $e$.

We may assume that $C$ is regular. Indeed, we may rotate $C$ by an angle $\theta$ where $\cos \theta$ and $\sin \theta$ are constructible,  to send it to a regular conic, because $C$ is in the field of constructible numbers. From now on, we assume that $C$ is regular.

Since both $C$ and $C_k$  are in $\F_k$ of the same eccentricity, $C_k$ can be obtained from $C$ by a magnification by a factor $\lambda >0$ and a translation in $a+bi$, where $a,b$ and $\lambda$ are in $\F_k$.  Therefore, the intersection point $z \in R_k \cap C_k$ can be obtained as follows:

(a) Translate the circle $R_k$ by a point in $\F_k$ and magnify  by a factor in $\F_k$  to get $R_k'$. 

(b) Intersect $R_k'$ with $C$ to get the corresponding intersection point $z'$.

(c) Reverse the process (a) to get $z$ from $z'$. 

Since we assumed that all of the points in $\F_k$ are $C$-constructible, the intersection $z' \in R_k' \cap C$ is a $C$-constructible point, because $R_k'$ is a circle in $\F_k$. We conclude that $z$ is a $C$-constructible point. 

Since every point in $\F_{k+1}$ is constructible from $\F_k$ and the intersection points  $z \in R_k \cap C_k$ obtained as in the above, it is $C$-constructible. This completes the proof.
\end{proof}


\vspace{0.5cm}
 
\noindent \textit{Acknowledgment.} This paper grew out from an   R\&E project ``A study on the numbers constructible from conic sections'', which was supported by Seoul Science High School during March-December 2012.

\vspace{1cm}

\noindent Insong Choe \\
Department of Mathematics, Konkuk University, 1 Hwayang-dong, Gwangjin-Gu, Seoul 143-701, Korea.\\
Email: \texttt{ischoe@konkuk.ac.kr}\\
\\
Seungjin Baek,  Yoonho Jung, Dongwook Lee and Junggyo Seo\\
Seoul Science High School, Uamgil 63 (Hyewha-dong 1-1), Jongro-ku, Seoul 110-530, Korea.\\
Emails: \texttt{bjh7790@naver.com}, \texttt{yoonhoim@nate.com}, \texttt{dwleeid@naver.com}, \texttt{sjgceo@naver.com}
\end{document}